\documentclass[12pt,a4paper]{article}

\usepackage[utf8]{inputenc}
\usepackage[T2A]{fontenc}
\usepackage[english]{babel}
\usepackage{amscd,amssymb,amsthm}
\usepackage{amsfonts,amsmath}
\usepackage{hyperref}
\usepackage{graphicx}

\newtheorem{claim}{Claim}
\newtheorem{theorem}{Theorem}

\newtheorem{question}{Question}
\newtheorem{definition}{Definition}

\newtheorem{conjecture}{Conjecture}

\title{The chromatic number of Euclidean space with dense color classes}

\author{Maxim Didin, Vsevolod Voronov}

\begin{document}

\maketitle

\begin{abstract}
In this note we construct colorings of Euclidean space
$\mathbb{R}^n$ with finitely many colors such that any two points at unit distance have different colors and, in addition, each color class is dense
in $\mathbb{R}^n$. In particular, 12 dense colors suffice to color $\mathbb{R}^2$. In arbitrary dimension, we show that
$n\chi(\mathbb{R}^n)+1$ colors suffice, where
$\chi(\mathbb{R}^n)$ denotes the chromatic number of $\mathbb{R}^n$ in the
standard formulation.
\end{abstract}

\section{Introduction}

We consider a version of the Hadwiger--Nelson--Erdős problem concerning
the chromatic number of subsets of Euclidean space, with the following
additional requirement: every neighborhood of every point must contain
points of every color. The book~\cite{Soifer} is devoted to the history of
this problem and its generalizations and variants. An overview of recent
results can be found in~\cite{arman2024upper,cherkashin2018chromatic}.

Let $M \subseteq \mathbb{R}^n$ be a subset of $n$-dimensional Euclidean
space equipped with the induced metric; $M$ is allowed to coincide with the
whole space. Let
\[
c\colon M\to\{1,\dots,k\}
\]
be a function defining a coloring with $k$ colors, and let
\[
C_i=\{x\in M\mid c(x)=i\},\qquad i=1,\dots,k,
\]
be the color classes determined by the coloring $c$.

\begin{definition}
A coloring is called proper if no two points of the same color are at unit
distance from each other.
\end{definition}

Equivalently, one may consider a proper coloring of the distance graph
$G(M)$, whose vertices are the points of $M$, with two vertices joined by an
edge whenever the corresponding points are at unit distance. 

\begin{definition}
A coloring of $M$ with $k$ colors and color classes
$C_1,\dots,C_k$ is called everywhere dense if each of the sets
$C_1,\dots,C_k$ is dense in $M$ with respect to the standard topology.
\end{definition}

\begin{definition}
The least $k$ for which there exists a proper everywhere dense coloring of
$M$ with $k$ colors is called the chromatic number of $M$ with
dense color classes. We denote this number by
$\chi_{ed}(M)$.
\end{definition}

Clearly, imposing additional restrictions on colorings cannot decrease the
chromatic number. Therefore, for every
$M\subseteq\mathbb{R}^n$ (and, in fact, for every metric space), we have
\[
\chi(M)\leq\chi_{ed}(M).
\]

It is also important to note that, when color classes are required to be
dense, as with any other topological or geometric restriction on
color classes, we cannot conclude that the chromatic number is attained by
some finite graph. In other words, the de Bruijn--Erdős theorem cannot be
applied directly to this restricted coloring problem.

Consider colorings of the sphere
\[
S^2_{r_0}\subset\mathbb{R}^3,\qquad r_0=\frac{\sqrt{2}}{2},
\]
equipped with the induced Euclidean metric. Since, for
$x,y\in S^2_{r_0}$, the condition $\|x-y\|=1$ implies that the vectors
$x$ and $y$ are orthogonal, in this setting we may equivalently consider
colorings in which no two vectors of the same color have zero scalar
product. For convenience, we assume that the center of the sphere is the
origin.

It is known from~\cite{kochen1967problem,simmons1976chromatic} that
\[
\chi(S^2_{r_0})=4.
\]
The upper bound is given by the octahedral coloring, that is, a coloring in
which the sphere is divided into eight spherical triangles obtained by
choosing the signs of the coordinates $x_1,x_2,x_3$. The interior of each
triangle is assigned one color, while the colors of boundary points may be
chosen, among the colors of the adjacent triangles, in infinitely many
ways. The colorings of $S^2_{r_0}$ were studied in~\cite{Holmsen16}, where
several sufficient conditions were obtained for a $4$-coloring of
$S^2_{r_0}$ to be octahedral. The question of whether there exist other
$4$-colorings, in particular, everywhere dense ones, remains open.

Apparently, $S^2_{r_0}$ is the first example of a subset of Euclidean space
for which a nontrivial upper bound on the chromatic number with
dense color classes was obtained.

\begin{claim}[see~\cite{Holmsen16}]
    \[
    \chi_{ed}(S^2_{r_0})\leq 9,\qquad
    r_0=\frac{\sqrt{2}}{2}.
    \]
\end{claim}

The construction is based on $p$-adic (non-Archimedean) absolute values
and on the fact that, for each coordinate, the set of points on the sphere
at which the absolute value of that coordinate is strictly greater than
the absolute values of the other coordinates provides an example of an
everywhere dense set containing no pair of points at unit distance. Indeed,
the scalar product of any two vectors from such a set is nonzero.

In the next section, we obtain upper bounds for
$\chi_{ed}(\mathbb{R}^n)$.

\section{Main result}

\begin{theorem}
Suppose that $\mathbb{R}^n$ admits a proper coloring with $k$ colors.
Then
\[
    \chi_{ed}(\mathbb{R}^n)\leq nk+1.
\]
\label{thm:main}
\end{theorem}

\begin{proof}
We shall say that the coordinates of the points of a set
$X\subset\mathbb{R}^n$ are algebraically independent if the union of the
sets of first coordinates, second coordinates, and so on up to the
$n$-th coordinates, in other words, the union of the projections of $X$ onto
the coordinate axes is algebraically independent over $\mathbb{Q}$.

Let $X$ be an everywhere dense subset of $\mathbb{R}^n$ whose coordinates
are algebraically independent over $\mathbb{Q}$. Observe that such a set exists. Indeed, if $X$ is finite
or countable, then the set of such points is contained in a countable union
of proper algebraic subsets of \(\mathbb R^n\). Hence it has empty
interior. In particular, every nonempty open ball contains a point that
can be added.  In particular, for every $\varepsilon>0$,
the construction can be carried out so that $X=X(\varepsilon)$ is an
$\varepsilon$-net in $\mathbb{R}^n$.

Let $A_1,\dots,A_k$ be the color classes in the original proper coloring
of $\mathbb{R}^n$. We shall construct a proper coloring of $\mathbb{R}^n$
with $nk+1$ dense color classes.

Partition $X$ into $n$ subsets
$X_1,\dots,X_n$, each dense in $\mathbb{R}^n$. Define
\[
    A'_{i,j}
    =
    X_j\cup
    \left(
        A_i\setminus
        \bigcup_{x\in X_j}S^{n-1}_1(x)
    \right),
    \qquad
    1\leq i\leq k,\quad 1\leq j\leq n,
\]
where $S^{n-1}_1(x)$ denotes the unit sphere centered at $x$.

Let
\[
    A^*
    =
    \bigcup_{x_j\in X_j,\;1\leq j\leq n}
    \bigcap_{1\leq j\leq n} S^{n-1}_1(x_j).
\]
Thus, $A^*$ is the set of all intersections of $n$ unit spheres whose
centers are chosen from $X_1,\dots,X_n$, one center from each set.

The set $A^*$ is everywhere dense. Indeed, given any point
$y\in\mathbb{R}^n$, choose points
\[
z_1,\dots,z_n\in S^{n-1}_1(y)
\]
such that the vectors $z_1-y,\dots,z_n-y$ are linearly independent.
If the points $x_j\in X_j$ are chosen sufficiently close to $z_j$, then
the corresponding $n$ spheres have an intersection point arbitrarily
close to $y$. This follows, for example, from the implicit function
theorem. Hence every neighborhood of $y$ intersects~$A^*$.

The sets
\[
    B_{i+n(j-1)}=A'_{i,j},
    \qquad 1\leq i\leq k,\quad 1\leq j\leq n,
\]
together with
\[
    B_{nk+1}=A^*
\]
form a cover of $\mathbb{R}^n$. Indeed, any point not covered by the sets
$A'_{i,j}$ belongs, by construction, to the intersection of $n$ unit
spheres, and hence belongs to~$A^*$.

Each of the sets $A'_{i,j}$ is unit distance free: the set $A_i$ is
unit distance free by assumption, points of $X_j$ cannot be at unit
distance from one another because of the algebraic independence of the
coordinates of $X$, and the spheres removed from $A_i$ eliminate pairs
consisting of a point of $X_j$ and a point of $A_i$. Moreover, the set
$A^*$ is unit distance free, again by the algebraic independence of the
coordinates of the centers.

It remains to pass from the cover by the everywhere dense
unit distance free sets
$B_1,\dots,B_{nk+1}$ to a partition into everywhere dense sets
$C_1,\dots,C_{nk+1}$. For example, choose countable, pairwise disjoint,
everywhere dense subsets
\[
    B'_i\subset B_i,\qquad 1\leq i\leq nk+1,
\]
and put $B'_i$ into $C_i$. Assign each remaining point to one of the
indices $i$ for which it belongs to $B_i$. Thus $C_i\subseteq B_i$ for
every $i$, so the resulting partition is proper, and every $C_i$ is
everywhere dense.
\end{proof}

To improve this bound, one may consider covers of $\mathbb{R}^n$ by
$k'$ unit distance free sets arranged in $n$ layers. If $k'<nk$, this
gives a better bound than the one in Theorem~\ref{thm:main}. In the planar
case, such a construction is known~\cite{parts2022upper}.

\begin{figure}
    \centering
    \includegraphics[width=14cm]{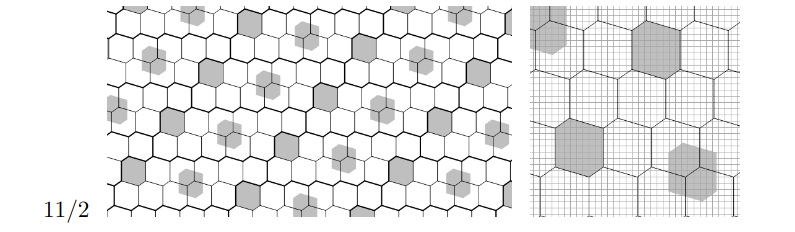}
    \caption{2-fold coloring with 11 colors found by J. Parts~\cite{parts2022upper}}
    \label{fig:2fold}
\end{figure}

\begin{claim}
    \[
        \chi_{ed}(\mathbb{R}^2)\leq 12.
    \]
\end{claim}

\begin{proof}
Let $A_1,\dots,A_{11}$ be unit distance free sets forming a
two-fold cover of the plane; that is, every point of the plane belongs
to at least two of these sets (see Fig.~\ref{fig:2fold}). Choose
everywhere dense sets
\[
    X_1,\dots,X_{11}\subset X
\]
and define
\[
    A'_i
    =
    X_i\cup
    \left(
        A_i\setminus
        \bigcup_{x\in X_i}S^1_1(x)
    \right),
    \qquad 1\leq i\leq 11.
\]
Let $A^*$, as above, be the set of all pairwise intersections of unit
circles with centers in $X$.

We obtain a cover of $\mathbb{R}^2$ by the everywhere dense
unit distance free sets
\[
    A'_1,\dots,A'_{11},A^*.
\]
Indeed, since the original cover is two-fold, every point not covered
by any of the sets $A'_i$ belongs to two unit circles with centers in
$X$, and therefore belongs to $A^*$. As above, this cover can be converted
into a proper coloring with 12 dense colors.
\end{proof}

\section{An arbitrarily large gap}

It is worth noting that, under the present requirement, the chromatic
number of the strip $\mathbb{R}\times[0,\varepsilon]$ is necessarily larger
than its ordinary chromatic number. In the standard setting, the latter is
equal to $3$ when $\varepsilon\leq\frac{\sqrt{3}}{2}$~\cite{bauslaugh1998tearing}.

\begin{claim}
    For every $\varepsilon>0$,
    \[
        \chi_{ed}\bigl(\mathbb{R}\times[0,\varepsilon]\bigr)\geq 4.
    \]
\end{claim}

\medskip
\emph{Sketch of the proof.} For sufficiently large $m=m(\varepsilon)$, the odd cycle
$C_{2m+1}$ admits a unit distance realization in the strip
$\mathbb{R}\times[0,\varepsilon]$ (see~\cite{bauslaugh1998tearing}). We may also
assume that none of its vertices lies on the boundary of the strip; this
can always be achieved by increasing $m$ if necessary.

Attach a pendant vertex to each vertex of the cycle in such a way that the
corresponding Jacobian, describing the dependence of the coordinates of the
cycle vertices on the coordinates of the pendant vertices, has full rank.
A similar construction was used in~\cite{DCGspheres}. Consequently, the
pendant vertices can be perturbed by arbitrarily small amounts so that they
all lie in a single color class, while the unit distance constraints are
preserved.

The vertices of the odd cycle are therefore forbidden from using that
color. Since an odd cycle requires at least three colors in a proper
coloring, at least four colors are required altogether.

\medskip

By the same idea, one can construct examples for which the gap between the
ordinary chromatic number and the chromatic number with dense
color classes is arbitrarily large.

\begin{claim}
    For every $\varepsilon>0$ and every positive integer $k$,
    \[
        \chi_{ed}\bigl(\mathbb{R}\times[0,\varepsilon]^k\bigr)\geq k+3.
    \]
\end{claim}

\medskip
\emph{Sketch of the proof.}  Indeed, attach $k$ pendant vertices to each vertex of an odd cycle and
consider a general-position unit distance realization of the resulting
graph in the ``box'' $\mathbb{R}\times[0,\varepsilon]^k$.
The realization can be chosen so that the pendant vertices can be perturbed
independently by arbitrarily small amounts while preserving all
unit distance constraints. If there were a coloring with at most $k+2$
dense color classes, the $k$ pendant vertices attached to each
cycle vertex could be perturbed into $k$ pairwise distinct color classes.
The cycle vertices would then be forbidden from using these $k$ colors, but
the remaining two colors cannot properly color an odd cycle. Hence at least
$k+3$ colors are necessary.

\medskip

On the other hand, if $\varepsilon$ is sufficiently small, we still have

\[
 \chi\bigl(\mathbb{R}\times[0,\varepsilon]^k\bigr) = 3.
\]

\section{Other norms}

In recent years, chromatic number problems for geometric graphs have been
studied extensively for distances induced by the $p$-norm and, more
generally, by arbitrary norms. We describe several extensions of the
problem considered here to distances in $\mathbb{R}^n$ induced by a norm.

For a real number $p\geq 1$, let
\[
    \|x\|_p=\left(\sum_{i=1}^n |x_i|^p\right)^{1/p}
\]
denote the $p$-norm.

In the limit as $p\to\infty$, we obtain
\[
    \|x\|_\infty=\max_{1\leq i\leq n}|x_i|.
\]

With dense color classes, this case remains straightforward.

\begin{claim}
    \[
    \chi_{ed}\bigl(\mathbb{R}^n;\|\cdot\|_\infty\bigr)
    =
    \chi\bigl(\mathbb{R}^n;\|\cdot\|_\infty\bigr)
    =
    2^n.
    \]
\end{claim}

\begin{proof}
Consider a proper coloring of the real line $\mathbb{R}$ with two
dense color classes $C_1$ and $C_2$. For every choice of
\[
(\tau(1),\tau(2),\dots,\tau(n)),\qquad \tau(i)\in\{1,2\},
\]
the set
\[
    C_{\tau(1)}\times C_{\tau(2)}\times\cdots\times C_{\tau(n)}
    \subset\mathbb{R}^n
\]
is dense in $\mathbb{R}^n$ and contains no pair of points at unit
$\ell_\infty$-distance. Indeed, if two points in this set were at
$\ell_\infty$-distance one, then their distance in at least one coordinate
would be one, contradicting the properness of the coloring of~$\mathbb{R}$.

There are $2^n$ such sets; they are pairwise disjoint and form a partition
of~$\mathbb{R}^n$. Hence
\[
\chi_{ed}\bigl(\mathbb{R}^n;\|\cdot\|_\infty\bigr)\leq 2^n.
\]
On the other hand, the $2^n$ vertices of the unit cube
$\{0,1\}^n$ form a clique in the unit distance graph for the
$\ell_\infty$-norm. Therefore
\[
\chi\bigl(\mathbb{R}^n;\|\cdot\|_\infty\bigr)\geq 2^n,
\]
and the reverse inequality follows from the product coloring above.
Since $\chi_{ed}\geq\chi$, both equalities follow.
\end{proof}

\begin{claim}
If $p\in\mathbb{Q}$, then
\[
    \chi_{ed}\bigl(\mathbb{R}^n;\|\cdot\|_p\bigr)
    \leq n\chi\bigl(\mathbb{R}^n;\|\cdot\|_p\bigr)+1.
\]
\end{claim}

\begin{proof}
It suffices to observe that, when $p$ is rational, the algebraic
independence of the coordinates prevents two points of the auxiliary set
from being at unit distance, exactly as in Theorem~\ref{thm:main}. The
remainder of the argument is entirely analogous.
\end{proof}

In~\cite{alon2025unit}, it was shown that, for a ``typical'' norm (in the
sense of Baire category in the space of norms) the chromatic number of
$\mathbb{R}^n$ is equal to~$2^n$. Since the coloring constructed in
\cite{alon2025unit} is obtained nonconstructively, it is natural to
conjecture that the additional requirement considered here does not change
the chromatic number in this case.

\begin{conjecture}
For a ``typical'' norm (see~\cite{alon2025unit}), we have
\[
    \chi_{ed}\bigl(\mathbb{R}^n;\|\cdot\|\bigr)
    =
    \chi\bigl(\mathbb{R}^n;\|\cdot\|\bigr)
    =
    2^n.
\]
\end{conjecture}

\section{Open questions}

 Suppose that the same collection of
colors is dense in the neighborhoods of two points at unit distance, while
no other colors are used in those neighborhoods. How many colors are
required?

\begin{question}
Let \(U_\varepsilon(u)\) denote the open ball of radius \(\varepsilon\)
centered at \(u\). Consider the set
\[
P_\varepsilon=U_\varepsilon(v_0)\cup U_\varepsilon(v_1)
\]
in the Euclidean plane, where
\[
v_0=(0,0),\qquad v_1=(1,0),\qquad 0<\varepsilon<\frac12.
\]
Determine \(\lim_{\varepsilon \to 0+} \chi_{ed}(P_\varepsilon)\).
\end{question}

Three colors are necessary to color \(P_\varepsilon\) with  dense
color classes, while four colors suffice for $0 < \varepsilon < \frac{1}{2}$. However, the question of whether a
3-coloring exists is more subtle than it may initially appear.

In the proof
of Theorem~\ref{thm:main}, we constructed unit distance free sets using the
algebraic independence of coordinates. Instead, one may ask how many
dense colors are needed for a proper coloring of a countable
union of circles whose centers have algebraically independent coordinates.

\begin{question}
Let \(X\subset\mathbb{R}^2\) be a set of points whose coordinates are algebraically independent over $\mathbb{Q}$. For each \(x\in X\), let
\[
\omega(x):=\{y\in\mathbb{R}^2:\|x-y\|=1\},
\]
and set
\[
W(X) =\bigcup_{x\in X}\omega(x).
\]
Is it true that
$
\chi_{\mathrm{ed}}(W(X))\leq 3?
$
\end{question}
 If the answer is affirmative, then the upper bound for $\chi_{ed}(\mathbb{R}^2)$ can be reduced to 10.

The gap between the known upper and lower bounds~\cite{cherkashin2018chromatic} for
\(\chi(\mathbb{R}^n)\) grows rapidly with \(n\). Thus, it is reasonable to expect that upper bounds may be obtained not only
from colorings associated with tilings of space by polytopes.

\begin{question}
In the Euclidean setting, or more generally for any fixed norm, can one improve the known upper bounds for
\(\chi(\mathbb{R}^n)\) by using dense color classes?
Does such an improvement occur for some particular subset of
\(\mathbb{R}^n\)?
\end{question}

\subsubsection*{Acknowledgements}

This work was basically done during LIPS’2025, a summer research program for students at MIPT. LLMs were used to edit the manuscript but were not used to obtain the results.

\bibliographystyle{abbrv}
\bibliography{main}

\end{document}